\documentclass[10pt]{amsart}

\usepackage[T1]{fontenc}
\usepackage[british]{babel}

\usepackage[a4paper,margin=3cm]{geometry}
\allowdisplaybreaks
\usepackage{amssymb}

\usepackage[
hyperfootnotes=false
]{hyperref}

\hypersetup{
    colorlinks=true,
    linkcolor=blue,
    citecolor=blue,
    urlcolor=red
}

\theoremstyle{plain}
\newtheorem{lemma}{Lemma}[section]
\newtheorem{theorem}[lemma]{Theorem}
\newtheorem{proposition}[lemma]{Proposition}
\newtheorem{corollary}[lemma]{Corollary}

\theoremstyle{definition}
\newtheorem{definition}[lemma]{Definition}

\theoremstyle{remark}
\newtheorem{remark}[lemma]{Remark}
\newtheorem{example}[lemma]{Example}

\numberwithin{equation}{section}

\newcommand{\R}{\mathbb R}

\begin{document}

\title[Intrinsic Harnack inequality]
{Harnack inequality for parabolic fractional
$p$-Laplace equations}

\begin{abstract}
We establish an intrinsic Harnack inequality with an optimal tail term for weak solutions of parabolic fractional $p$-Laplace equations with $p>2$
and $s\in(0,1)$ whose kernels are symmetric, measurable, and
comparable to $|x-y|^{-n-sp}$. It constitutes the nonlocal analogue
of the intrinsic Harnack inequality of DiBenedetto, Gianazza and Vespri.
In contrast to the exponential change of variables  and expansion-of-positivity used in the local case,
our technique is nonlocal and relies on comparison arguments and barriers in a variational
framework. Indeed, we obtain the stronger conclusion where the forward comparison starts immediately after the reference time,
without an additional time gap, which is false in the local case. The latter leads
to elliptic-type Harnack estimates for parabolic equations. Our result is new even in the linear case.

\end{abstract}

\keywords{Nonlocal parabolic equations, fractional $p$-Laplacian, intrinsic
Harnack inequality, tail estimates, measurable kernels}

\subjclass[2020]{Primary 35B45, 35R11; Secondary 35K65, 35B65}

\author[H.~Prasad]{Harsh Prasad}
\address[H.~Prasad]{Fakult\"at f\"ur Mathematik, Universit\"at Bielefeld,
33615 Bielefeld, Germany.}
\email[]{hprasad@math.uni-bielefeld.de}

\maketitle
\setcounter{tocdepth}{1}
\tableofcontents

\section{Introduction}

We study local weak solutions of the nonlocal parabolic equation
\[
\partial_tu+\mathcal L_tu=0,
\]
where $\mathcal L_t$ is an operator of fractional $p$-Laplace type,
with $p>2$ and $s\in(0,1)$. The kernel is symmetric and measurable,
may depend on time, and is comparable to $|x-y|^{-n-sp}$.
We establish a nonlocal analogue of the intrinsic Harnack inequality
of DiBenedetto, Gianazza and Vespri \cite{DGVActa,DGVDuke},
with an optimal time-integrated tail term; see Theorem \ref{thm:main}.
Our proof uses comparison with time-dependent barriers in place of
expansion of positivity and exploits the nonlocal interactions
as a source of positivity.

Harnack inequalities and H\"older estimates are central to the local
regularity theory of elliptic and parabolic equations. Early nonlocal
Harnack inequalities were obtained for harmonic functions associated
with stable-like jump processes in \cite{BassLevin}. A variational
approach to local regularity for integro-differential operators with
measurable kernels was developed in \cite{KassmannRegularity}.
In the nonlinear elliptic setting, local boundedness and H\"older
continuity for fractional $p$-Laplacian were established in
\cite{DKPRegularity}, and a Harnack inequality in \cite{DKPHarnack}.
A unified treatment of regularity and Harnack estimates through
fractional De Giorgi classes was given in \cite{Cozzi}; see also
\cite{APTElliptic} for an alternative proof of H\"older regularity
based on a continuous iteration.

In the linear nonlocal parabolic setting, probabilistic methods
yielded parabolic Harnack inequalities and two-sided heat kernel
estimates for stable-like processes
\cite{ChenKumagai}. Stability and characterisations of parabolic
Harnack inequalities for symmetric nonlocal Dirichlet forms on
metric measure spaces were developed in \cite{ChenKumagaiWang}.
On the analytic side, H\"older estimates for parabolic integral
equations were established in \cite{CaffarelliChanVasseur}, while
\cite{FelsingerKassmann} proved robust weak Harnack inequality and H\"older estimates. Further
parabolic Harnack estimates were obtained in
\cite{StromqvistHarnack}. For local weak solutions with merely
measurable, time-dependent coefficients, a full Harnack inequality
with time-integrated tails was established in
\cite{KassmannWeidner}. 
The distinction between local and global solutions is also important
for the time geometry. Time-insensitive Harnack estimates, with
overlapping comparison intervals, hold for global solutions of
linear nonlocal equations \cite{LiaoWeidner}; the same work shows
that such estimates fail in general when the equation is imposed
only locally. 

For parabolic fractional $p$-Laplace equations, local boundedness
was obtained in \cite{Stromqvist,DingZhangZhou, PrasadTewary}.  H\"older regularity for kernels
that are merely measurable and comparable to the fractional kernel
was developed in \cite{APT,LiaoHolder}.

The integrability in time of the tail is an important distinction
between regularity results. Boundedness and H\"older estimates in
\cite{ByunKim} require a tail condition weaker than a
supremum-in-time bound, but stronger than time-integrability of the
spatial tail density. Under the latter condition, a modulus of
continuity for locally bounded solutions was established in
\cite{LiaoModulus}, together with a H\"older modulus under a stronger
condition. Local boundedness estimates involving only the
time-integrated tail were obtained in \cite{KumagaiWangZhang}.
Natural-energy local boundedness under this tail condition was
proved on the Heisenberg group in \cite{KarTewary}; the argument
also applies in the Euclidean setting, as explained after
Theorem \ref{thm:main}. The distinction between Harnack and H\"older
estimates is substantive: already in the linear case, an
$L^1$-in-time tail is sufficient for the former but does not in
general yield the latter \cite{KassmannWeidner}.

Further regularity results concern related nonlinear equations.
A weak Harnack inequality for the nonlocal Trudinger equation was
proved in \cite{PrasadWeakHarnack}, while local H\"older regularity was established in
\cite{AdimurthiTrudinger}. Expansion of positivity for a broader
family of doubly nonlinear fractional equations was obtained in
\cite{MisawaYamaura}. For nonlocal doubly degenerate parabolic
equations, local H\"older continuity
was established in \cite{LiDoublyDegenerate}. Local H\"older
regularity for bounded weak solutions in a mixed singular-degenerate
range was obtained in \cite{AdimurthiModasiya}.
For porous medium-type equations
on bounded domains, global Harnack inequalities and interior and
boundary regularity were established in \cite{BonforteFigalliVazquez}.
Local H\"older continuity for nonlocal porous medium and fast diffusion equations with bounded measurable kernels was established in
\cite{KimLeePrasad}. 

For degenerate parabolic equations, the geometry of the Harnack
inequality also depends on the size of the solution. In the local
$p$-Laplace theory, this dependence is expressed through intrinsic
scaling and a waiting time proportional to $k^{2-p}r^p$, where $k$
is the positive value being compared and $r$ is the spatial radius.
We refer to \cite{DiBenedetto,DGV} for this theory.

In this paper, we establish an intrinsic Harnack inequality for
parabolic fractional $p$-Laplace equations in the degenerate range
$p>2$, with bounded measurable kernels. The estimate compares a
positive value $k=u(x_0,t_0)$ with the essential infimum on
$B_r(x_0)\times(t_0,t_0+2\tau)$, where $\tau=k^{2-p}r^{sp}$.
There is no additional forward time gap, but the assumptions retain
a backward intrinsic history interval.
It applies to local energy weak solutions under the time-integrated
tail condition without additional
local integrability or a supremum-in-time tail bound. The result thus
extends the intrinsic Harnack principle
\cite{DGVActa,DGVDuke} to purely nonlocal degenerate diffusion
and strengthens it in the nonlocal case. Indeed, we get elliptic type 
Harnack estimates for parabolic equations which are false in the local case. 

Beyond the estimate itself, the paper develops a variational
comparison approach to intrinsic Harnack inequalities that exploits
nonlocality as a source of positivity. It accommodates both nonlinear
degeneracy and merely measurable, time-dependent kernels, and avoids
an expansion-of-positivity argument for the original solution.
Our use of favourable nonlocal interactions to drive a time-dependent
barrier has precedents in the nonvariational arguments of
\cite{SilvestreDCDS,SilvestreAdvection}.
The method also provides alternative proofs of the linear parabolic
Harnack inequality, as explained in Appendix \ref{app:linear}, and,
by restriction to time-independent solutions, of the elliptic Harnack
inequality for $p>2$. As mentioned above, we are able to get
much stronger results via our method than those available earlier. 
We expect that our variational comparison approach will also be useful in establishing Harnack inequalities 
and other regularity results for broader classes of nonlocal equations. 
\\

\noindent\textbf{LLM Declaration}
During the preparation of this manuscript, I used ChatGPT and Claude
to assist with literature searches, the drafting and revision of the
exposition, and the exploration, drafting, and verification of parts of
the arguments. I have independently checked all mathematical statements
and proofs and take responsibility for the contents of the manuscript.

\section{Setting and Main Theorem}

\subsection{Notation and function spaces}

Throughout the paper, $n\ge1$, $p>2$, and $s\in(0,1)$. We set

\[
    N:=n+sp,
    \qquad
    F(a):=|a|^{p-2}a.
\]

For $x_0\in\R^n$ and $r>0$, $B_r(x_0)$ denotes the open Euclidean ball
with centre $x_0$ and radius $r$; we write $B_r:=B_r(0)$. If $E$ is a
measurable set, then $|E|$ denotes its Lebesgue measure. The notation
$E\Subset G$ means that $\overline E$ is a compact subset of the open set
$G$. For a real-valued function $v$, we use

\[
    v_+:=\max\{v,0\},
    \qquad
    v_-:=\max\{-v,0\}.
\]

For $z_0=(x_0,t_0)$ and $\rho,\theta>0$, we use the cylinders
\[
\begin{split}
Q^-_\rho(z_0;\theta)&:=B_\rho(x_0)\times(t_0-\theta\rho^{sp},t_0],\\
Q^+_\rho(z_0;\theta)&:=B_\rho(x_0)\times(t_0,t_0+\theta\rho^{sp}].
\end{split}
\]

For an open set $D\subset\R^n$, we write
\[
[v]_{W^{s,p}(D)}^p
:=\iint_{D\times D}
\frac{|v(x)-v(y)|^p}{|x-y|^{n+sp}}\,dx\,dy.
\]
The fractional Sobolev space $W^{s,p}(D)$ is equipped with the norm
\begin{equation}\label{eq:Wsp-norm}
\|v\|_{W^{s,p}(D)}^p
:=
\int_D|v|^p\,dx
+[v]_{W^{s,p}(D)}^p.
\end{equation}
We write $W^{s,p}_{\mathrm{loc}}(D)$ for the functions belonging to
$W^{s,p}(D')$ for every $D'\Subset D$. For every measurable set
$A\subset\R^n$, we define
\[
    W^{s,p}_0(A)
    :=
    \left\{
    v\in W^{s,p}(\R^n):
    v=0\ \text{a.e. in }\R^n\setminus A
    \right\}.
\]

All space-time function spaces below are Bochner spaces. In particular,
$L^q_{\mathrm{loc}}(I;X)$ means $L^q(J;X)$ for every compact interval
$J\Subset I$, and $C_{\mathrm{loc}}(I;X)$ has the analogous meaning.

Constants denoted by $C$ may change from line to line. Unless additional
dependence is stated, they depend only on the data
$n,p,s,\lambda,\Lambda$.

\subsection{Tail spaces and tail quantities}

We use the weighted tail space
\begin{equation}\label{eq:tail-space}
L^{p-1}_{sp}(\R^n)
:=
\left\{
v\in L^{p-1}_{\mathrm{loc}}(\R^n):
\|v\|_{L^{p-1}_{sp}(\R^n)}^{p-1}
:=
\int_{\R^n}
\frac{|v(y)|^{p-1}}{1+|y|^{n+sp}}\,dy
<\infty
\right\}.
\end{equation}
For a measurable function $v$, a bounded interval $J$ with $|J|>0$,
$x_0\in\R^n$, and $R>0$, we define
\begin{equation}\label{eq:tail-def}
\operatorname{Tail}_{p-1}(v;x_0,R,J)
:=
\left[
\frac{R^{sp}}{|J|}
\int_J\int_{\R^n\setminus B_R(x_0)}
\frac{|v(y,t)|^{p-1}}{|y-x_0|^{n+sp}}\,dy\,dt
\right]^{1/(p-1)}.
\end{equation}
In particular, the negative tail is obtained by taking $v=u_-$.

\subsection{Weak solutions}

Let $I\subset\R$ be an interval. The kernel
$K:\R^n\times\R^n\times I\to[0,\infty]$ is assumed to be measurable and,
for almost every $(x,y,t)$ with $x\ne y$, to satisfy
\begin{equation}\label{eq:kernel}
K(x,y,t)=K(y,x,t),
\qquad
\lambda |x-y|^{-N}
\le K(x,y,t)
\le \Lambda |x-y|^{-N}.
\end{equation}

The associated operator is
\begin{equation}\label{eq:operator}
\mathcal L_t v(x)
:={\rm P.V.}\int_{\mathbb R^n}
F\bigl(v(x)-v(y)\bigr)K(x,y,t)\,dy.
\end{equation}

The equation under consideration is
\begin{equation}\label{eq:main-equation}
\partial_tu+\mathcal L_tu=0.
\end{equation}

For functions for which the expression is finite, we define
\begin{equation}\label{eq:energy}
\mathcal{E}_t(v,\varphi)
:=
\frac12
\iint_{\R^n\times\R^n}
F\bigl(v(x)-v(y)\bigr)
\bigl(\varphi(x)-\varphi(y)\bigr)
K(x,y,t)\,dx\,dy.
\end{equation}
By symmetry of $K$, this normalisation makes $\mathcal E_t(v,\varphi)$
the variational pairing of $\mathcal L_t v$ with $\varphi$.

\begin{definition}[Local weak solution]\label{def:weak}
Let $\Omega\subset\R^n$ be open and let $I\subset\R$ be an open interval.
A measurable function $u:\R^n\times I\to\R$ with
\begin{equation}\label{eq:weak-regularity}
u\in
C_{\mathrm{loc}}\bigl(I;L^2_{\mathrm{loc}}(\Omega)\bigr)
\cap
L^p_{\mathrm{loc}}
\bigl(I;W^{s,p}_{\mathrm{loc}}(\Omega)\bigr),
\end{equation}
is a local weak solution of \eqref{eq:main-equation} in $\Omega\times I$ if 
for every compact set $\mathcal K\subset\Omega$ and every subinterval
$[t_1,t_2]\Subset I$,
\begin{equation}\label{eq:weak-tail-condition}
\int_{t_1}^{t_2}\int_{\R^n}
\frac{|u(x,t)|^{p-1}}{1+|x|^{n+sp}}\,dx\,dt
<\infty
\end{equation}
and
\begin{equation}\label{eq:weak-form}
\left.
\int_{\mathcal K}u(x,t)\varphi(x,t)\,dx
\right|_{t=t_1}^{t=t_2}
-
\int_{t_1}^{t_2}\int_{\mathcal K}
u\,\partial_t\varphi\,dx\,dt
+
\int_{t_1}^{t_2}\mathcal{E}_t(u,\varphi)\,dt
=0
\end{equation}
for every testing function
\begin{equation}\label{eq:weak-test-space}
\varphi\in
W^{1,2}_{\mathrm{loc}}\bigl(I;L^2(\mathcal K)\bigr)
\cap
L^p_{\mathrm{loc}}\bigl(I;W^{s,p}_0(\mathcal K)\bigr).
\end{equation}
\end{definition}

A function satisfying \eqref{eq:weak-regularity} and
\eqref{eq:weak-tail-condition} is a local weak supersolution if the left-hand
side of \eqref{eq:weak-form} is nonnegative for every nonnegative testing
function in \eqref{eq:weak-test-space}.

\subsection{Main theorem}

Whenever a locally bounded weak solution is evaluated at a point, we
use its locally continuous representative, whose existence follows
from Theorem \ref{thm:liao-modulus}. The same convention applies in
Appendix \ref{app:linear}, where the corresponding continuity result
is used at $p=2$.

For the main theorem, we use the following intrinsic parameters.
Given $r>0$ and a point $(x_0,t_0)$ with $u(x_0,t_0)>0$, we set
\[
k:=u(x_0,t_0),\qquad
\tau:=k^{2-p}r^{sp},\qquad
J:=(t_0-\tau,t_0+2\tau).
\]

\begin{theorem}[Intrinsic Harnack inequality]\label{thm:main}
Assume $p>2$, $s\in(0,1)$, and \eqref{eq:kernel}. Let $u$ be a locally
bounded local weak solution of \eqref{eq:main-equation} with
$u(x_0,t_0)>0$. With the intrinsic parameters $k,\tau,J$ above,
we assume that the solution domain contains
\[
    B_{4r}(x_0)\times(t_0-2\tau,t_0+3\tau).
\]

We further assume that, for some $R\ge4r$,

\[
    u\ge0
    \quad\text{a.e. in}\quad
    B_R(x_0)\times(t_0-2\tau,t_0+3\tau).
\]

Then
\begin{equation}\label{eq:main-harnack}
u(x_0,t_0)
\le
C\,
\operatorname*{ess\,inf}_{B_r(x_0)\times(t_0,t_0+2\tau)}u
+
C
\left(\frac rR\right)^{\frac{sp}{p-1}}
\operatorname{Tail}_{p-1}(u_-;x_0,R,J),
\end{equation}
where $C$ depends only on $n,p,s,\lambda,\Lambda$.
\end{theorem}

\begin{remark}
\label{rem:no-forward-gap}
The forward comparison has no additional time gap.
The interval in \eqref{eq:main-harnack} begins at $t_0$, rather than
at $t_0+\tau$. 
\end{remark}

\begin{corollary}[Elliptic-type Harnack inequality]
\label{cor:elliptic-type}
Under the assumptions of Theorem \ref{thm:main}, with
\[
k=u(x_0,t_0)>0,\qquad
\tau=k^{2-p}r^{sp},\qquad J=(t_0-\tau,t_0+2\tau),
\]
we have the same-time spatial estimate
\begin{equation}\label{eq:same-time-harnack}
u(x_0,t_0)
\le C\inf_{B_r(x_0)}u(\cdot,t_0)
+C\left(\frac rR\right)^{sp/(p-1)}
\operatorname{Tail}_{p-1}(u_-;x_0,R,J).
\end{equation}
Here $C$ depends only on $n,p,s,\lambda,\Lambda$. In particular, if
$u\ge0$ on $\R^n\times(t_0-2\tau,t_0+3\tau)$, then
\[
u(x_0,t_0)\le C\inf_{B_r(x_0)}u(\cdot,t_0).
\]
\end{corollary}

\begin{proof}
Let $m$ be the essential infimum over the forward cylinder in
\eqref{eq:main-harnack}. By continuity, $u(x,t)\ge m$ for every
$x\in B_r(x_0)$ and $t\in(t_0,t_0+2\tau)$. Letting $t\downarrow t_0$
gives $m\le\inf_{B_r(x_0)}u(\cdot,t_0)$. The conclusion follows from
\eqref{eq:main-harnack}. Global nonnegativity makes the tail term vanish.
\end{proof}

\begin{corollary}[Globally nonnegative solutions]\label{cor:global}
Under the assumptions of Theorem \ref{thm:main}, if

\[
    u\ge0
    \quad\text{on }\R^n\times(t_0-2\tau,t_0+3\tau),
\]

then
\begin{equation}\label{eq:global-harnack}
u(x_0,t_0)
\le
C\,
\operatorname*{ess\,inf}_{B_r(x_0)\times(t_0,t_0+2\tau)}u.
\end{equation}
\end{corollary}
\begin{remark}
The local boundedness argument in
\cite[Theorem~1.4]{KarTewary}, developed on the Heisenberg group,
also applies in the Euclidean setting, with the corresponding
fractional Sobolev and interpolation inequalities. For $p>2$, the
local quantities in their estimate are finite under
\eqref{eq:weak-regularity}, while the time-integrated tail is finite
by \eqref{eq:weak-tail-condition}. Applying this argument to $u$ and
$-u$, we see that every local weak solution in the sense of
Definition \ref{def:weak} is locally essentially bounded. Thus the
local boundedness assumption in Theorem \ref{thm:main} is automatic;
no additional local integrability or supremum-in-time tail bound
is required.
\end{remark}
\begin{example}[A local travelling wave]\label{ex:local-p-wave}
The absence of an additional forward time gap has no analogue for
the local parabolic $p$-Laplacian when $p>2$, even if the solution
exists for all times. Set
\[
q:=\frac{p-1}{p-2},\qquad
v(x,t):=q^{-q}(x_1+t)_+^q.
\]
Then $v_t=v_{x_1}$ and
$|v_{x_1}|^{p-2}v_{x_1}=v$, so
\[
v_t-\operatorname{div}(|Dv|^{p-2}Dv)=0
\quad\text{in }\R^n\times\R.
\]
These identities hold weakly across $x_1+t=0$ as well, since $q>1$
and both $v$ and its flux are continuous there.
Choose $x_0=d e_1$, where $e_1$ is the first coordinate vector and
$0<d<r$, and put $t_0=0$. We have $v(x_0,0)>0$. For each
$0<t<r-d$, the set $B_r(x_0)\cap\{x_1<-t\}$ has positive measure,
and $v(\cdot,t)$ vanishes there. Hence, with
$\tau=v(x_0,0)^{2-p}r^p$,
\[
\operatorname*{ess\,inf}_{B_r(x_0)\times(0,2\tau)}v=0.
\]
Thus the local analogue of the forward comparison without a time
gap fails, despite the availability of an arbitrarily long backward
history. The scale $v(x_0,0)^{2-p}r^p$ remains the intrinsic scale
of the classical Harnack theory \cite{DGVActa,DGV}; its positive
waiting time cannot in general be removed. Example
\ref{ex:local-heat} gives a separate obstruction for the local heat
equation.
\end{example}

% =============================================================================

\section{Preliminaries}

\subsection{A scalar inequality}

We begin with an inequality for $F$.

\begin{lemma}\label{lem:scalar}
For $p\ge2$ and $a\ge b$,
\begin{equation}\label{eq:scalar}
F(a)-F(b)
\ge
2^{2-p}(a-b)^{p-1}.
\end{equation}
\end{lemma}

\subsection{A weak comparison principle}

We record the following comparison principle used below. 

\begin{lemma}[Weak comparison]\label{lem:weak-comparison}
Let $p\ge2$, let $D\subset\R^n$ be a bounded open set, and let
$a<b$. Set $V:=W^{s,p}_0(D)$ and $p':=p/(p-1)$, and assume that
$K$ satisfies \eqref{eq:kernel}. Let $v,w:\R^n\times[a,b]\to\R$
and constants $c_v\le c_w$ satisfy
\[
v-c_v,\ w-c_w\in L^p((a,b);V),
\qquad
v|_D,\ w|_D\in C([a,b];L^2(D)),
\]
and
\[
\partial_t(v|_D),\ \partial_t(w|_D)
\in L^{p'}((a,b);V^*).
\]
In particular, $v=c_v\le c_w=w$ almost everywhere on
$(\R^n\setminus D)\times(a,b)$.
Suppose that $v(\cdot,a)\le w(\cdot,a)$ almost everywhere in $D$ and
that, for almost every $t\in(a,b)$ and every nonnegative $\varphi\in V$,
\begin{equation}\label{eq:weak-comparison-assumption}
\langle\partial_t(v-w),\varphi\rangle_{V^*,V}
+\mathcal E_t(v,\varphi)-\mathcal E_t(w,\varphi)\le0.
\end{equation}
Then $v(\cdot,t)\le w(\cdot,t)$ almost everywhere in $D$ for every
$t\in[a,b]$.
\end{lemma}

\iffalse 
\begin{proof}
Set $d:=v-w$ and $z:=d_+$. Since $c_v\le c_w$, the zero extension
of $z|_D$ agrees with $z$ and
\[
z\in L^p((a,b);V)\cap C([a,b];L^2(D)).
\]
Indeed, if $q:=(v-c_v)-(w-c_w)$, then $q\in L^p((a,b);V)$ and
$z=(q-(c_w-c_v))_+$, a truncation that preserves $V$.
The energy-space chain rule for this truncation, justified by time
regularisation, allows us to test \eqref{eq:weak-comparison-assumption}
with $z$ and integrate to obtain
\begin{equation}\label{eq:weak-comparison-energy}
\frac12\|z(t)\|_{L^2(D)}^2
+\int_a^t\bigl[\mathcal E_\tau(v,z)-\mathcal E_\tau(w,z)\bigr]\,d\tau
\le\frac12\|z(a)\|_{L^2(D)}^2=0.
\end{equation}
For completeness, the chain rule applies to
$q\mapsto\frac12\|(q-(c_w-c_v))_+\|_{L^2(D)}^2$, since
$q\in L^p((a,b);V)$ and $\partial_tq\in L^{p'}((a,b);V^*)$.

For almost every $x,y$ and $\tau$, the quantities
\[
F(v(x)-v(y))-F(w(x)-w(y))
\quad\text{and}\quad z(x)-z(y)
\]
have the same sign whenever the latter is nonzero: both signs are
determined by $d(x)-d(y)$. Here we have suppressed the time variable.
Thus monotonicity of $F$ and of the positive-part map, together with
$K\ge0$, gives
\[
\mathcal E_\tau(v,z)-\mathcal E_\tau(w,z)\ge0.
\]
It follows from \eqref{eq:weak-comparison-energy} that $z(t)=0$;
time continuity gives the conclusion for every $t\in[a,b]$.
\end{proof}

\fi 
\subsection{A modulus of continuity}

For a locally bounded local weak solution $u$ in $\Omega\times I$ and
$Q^-_{\widetilde R}(z_0;1)\Subset\Omega\times I$, we write
\begin{equation}\label{eq:liao-omega}
\begin{split}
\omega
:={}&
2\operatorname*{ess\,sup}_{Q^-_{\widetilde R}(z_0;1)}|u|
\\
&+
\int_{t_0-\widetilde R^{sp}}^{t_0}
\int_{\R^n\setminus B_{\widetilde R}(x_0)}
\frac{|u(y,t)|^{p-1}}{|y-x_0|^{N}}\,dy\,dt.
\end{split}
\end{equation}
We record the following result from \cite[Theorem~1.1]{LiaoModulus}.

\begin{theorem}\label{thm:liao-modulus}
Let $u$ and $Q^-_{\widetilde R}(z_0;1)$ be as above, with $\omega$
given by \eqref{eq:liao-omega}. There exist constants $C>1$ and
$\beta,\sigma\in(0,1)$, depending only on the data, with the following
property.
If $\omega>0$, $0<r<R<\widetilde R$, and
\[
Q^-_R(z_0;\omega^{2-p})
\subset Q^-_{\widetilde R}(z_0;1),
\]
then
\begin{equation}\label{eq:liao-modulus}
\begin{split}
\operatorname*{ess\,osc}_{Q^-_{\sigma r}(z_0;\omega^{2-p})}u
\le{}&
2\omega\left(\frac rR\right)^\beta
\\
&+C
\int_{t_0-\omega^{2-p}(rR)^{sp/2}}^{t_0}
\int_{\R^n\setminus B_{\widetilde R}(x_0)}
\frac{|u(y,t)|^{p-1}}{|y-x_0|^N}\,dy\,dt.
\end{split}
\end{equation}
In particular, $u$ has a continuous representative in $\Omega\times I$.
\end{theorem}

\subsection{Forward propagation of positivity}

We record the globally nonnegative version of the forward-in-time De Giorgi
lemma from \cite[Lemma~3.11]{APT}.

\begin{lemma}\label{lem:APT-forward}
Let $u$ be a local weak supersolution of \eqref{eq:main-equation} in a
neighbourhood of
\[
Q^+_{2\rho}(z_0;\theta),
\]
and assume that
\begin{equation}\label{eq:APT-global-nonnegative}
u\ge0
\quad\text{a.e. in }
\R^n\times(t_0,t_0+\theta(2\rho)^{sp}].
\end{equation}
Suppose that, for some $M>0$ and $\xi\in(0,1)$,
\begin{equation}\label{eq:APT-initial-positive}
u(\cdot,t_0)\ge\xi M
\quad\text{a.e. in }B_{2\rho}(x_0).
\end{equation}
For each $a\in(0,1)$, there exists $\nu_0\in(0,1)$, depending only on $a$ and
the data, such that if
\begin{equation}\label{eq:APT-small-measure}
\frac{
\bigl|\{u<\xi M\}\cap Q^+_{2\rho}(z_0;\theta)\bigr|
}{
\bigl|Q^+_{2\rho}(z_0;\theta)\bigr|
}
\le
\nu_0
\frac{(\xi M)^{2-p}}{\theta},
\end{equation}
then
\begin{equation}\label{eq:APT-positive-conclusion}
u\ge a\xi M
\quad\text{a.e. in }
B_\rho(x_0)\times(t_0,t_0+\theta(2\rho)^{sp}].
\end{equation}
\end{lemma}

\section{Auxiliary Results}

\subsection{Positivity from a normalised point value}

We begin with the following consequence of Theorem
\ref{thm:liao-modulus}. 

\begin{lemma}\label{lem:continuity}
There exist constants

\[
    \delta_0>0,
    \qquad
    \rho_0\in(0,1),
    \qquad
    c_0\in(0,1),
\]

depending only on the data, with the following property. Let $v$ be a locally
bounded weak solution in

\[
    B_4\times(-4^{sp},0]
\]

such that

\[
    0\le v\le2
    \quad\text{in }B_4\times(-4^{sp},0],
    \qquad
    v(0,0)=1.
\]

If
\begin{equation}\label{eq:normalized-tail-small}
\int_{-4^{sp}}^0
\int_{\R^n\setminus B_4}
\frac{|v(y,t)|^{p-1}}{|y|^{n+sp}}\,dy\,dt
\le\delta_0,
\end{equation}
then
\begin{equation}\label{eq:normalized-positive}
v\ge c_0
\end{equation}
in a fixed backward cylinder

\[
    B_{\rho_0}\times(-\rho_0^{sp},0].
\]

\end{lemma}

\begin{proof}
We apply Theorem \ref{thm:liao-modulus} with $z_0=(0,0)$,
$\widetilde R=3$, and $R=2$. Denote the quantity in
\eqref{eq:liao-omega} by $\omega_3$. We use the continuous
representative of $v$, for which $v(0,0)=1$ by hypothesis. Hence,
using $0\le v\le2$ in the local cylinder and
\eqref{eq:normalized-tail-small}, we have, provided $\delta_0\le1$,
\begin{equation}\label{eq:omega3-bounds}
2
\le
\omega_3
\le
4
+
2^{p-1}3^{sp}
\int_{B_4\setminus B_3}|y|^{-N}\,dy
+
\delta_0
\le
\omega_*,
\end{equation}
where $\omega_*>1$ depends only on the data. In particular,
\[
Q^-_2(0;\omega_3^{2-p})
\subset Q^-_3(0;1),
\]
since $p>2$ and $\omega_3\ge2$.

Let $0<r<1$. Splitting the exterior term in
\eqref{eq:liao-modulus} into $B_4\setminus B_3$ and
$\R^n\setminus B_4$, and using again $0\le v\le2$ in the former region,
we obtain
\begin{equation}\label{eq:normalized-oscillation-bound}
\begin{split}
\operatorname*{ess\,osc}_
{Q^-_{\sigma r}(0;\omega_3^{2-p})}v
&\le
2\omega_*\left(\frac r2\right)^\beta
+C\omega_3^{2-p}(2r)^{sp/2}
+C\delta_0
\\
&\le
2\omega_*\left(\frac r2\right)^\beta
+Cr^{sp/2}
+C\delta_0.
\end{split}
\end{equation}
Choose $r_*>0$, depending only
on the data, so small that
\[
2\omega_*\left(\frac {r_*}2\right)^\beta
+Cr_*^{sp/2}
\le\frac14,
\]
and then choose $\delta_0\in(0,1]$ so that $C\delta_0\le1/4$.
It follows that
\begin{equation}\label{eq:normalized-small-oscillation}
\operatorname*{ess\,osc}_
{Q^-_{\sigma r_*}(0;\omega_3^{2-p})}v
\le\frac12.
\end{equation}

Finally, set
\[
\rho_0
:=
\frac{\sigma r_*}{2}
\omega_*^{(2-p)/(sp)}.
\]
Then $\rho_0\in(0,1)$ and, because $\omega_3\le\omega_*$ and $p>2$,
\[
B_{\rho_0}\times(-\rho_0^{sp},0]
\subset
Q^-_{\sigma r_*}(0;\omega_3^{2-p}).
\]
Combining this inclusion with \eqref{eq:normalized-small-oscillation} and
$v(0,0)=1$ yields $v\ge1/2$ in the smaller cylinder. Thus the conclusion
holds with $c_0=1/2$.
\end{proof}

\subsection{A point to moment estimate}

Throughout the rest of the article, we use the weight
\[
    \omega(y):=(1+|y|)^{-N},
\]
and, for a measurable function $u$, the moments
\begin{equation}\label{eq:weighted-moments}
M_\pm(t):=\int_{\R^n}u_\pm(y,t)^{p-1}\omega(y)\,dy.
\end{equation}
These moments are locally integrable in time under
\eqref{eq:weak-tail-condition}. We now show that positivity at a point
forces a positive weighted moment at earlier times.

\begin{proposition}\label{prop:moment}
There exists $\varepsilon_0>0$, depending only on the data, such that the
following holds. Suppose $u$ is a locally bounded weak solution in
$B_2\times(-1,1)$ and assume that

\[
    u\ge0
    \quad\text{a.e. in }B_2\times(-1,0].
\]

Suppose moreover that

\[
    u(0,0)=1.
\]

Then
\begin{equation}\label{eq:absolute-moment}
\int_{-1}^0
\int_{\R^n}
|u(y,t)|^{p-1}\omega(y)\,dy\,dt
\ge\varepsilon_0.
\end{equation}
\end{proposition}

\begin{proof}
We set
\[
\mathcal M
:=
\int_{-1}^0\int_{\R^n}
|u(y,t)|^{p-1}\omega(y)\,dy\,dt.
\]
There is nothing to prove if $\mathcal M=\infty$, so there is no loss of generality in assuming that 
it is finite. We first select a point and an intrinsic cylinder on which the value
of $u$ can at most double.

Indeed, let $c\in(0,1/8)$ be a constant, depending only on the data, sufficiently
small as specified below. We run a stopping time argument. We start with
\[
z_0=(x_0,t_0)=(0,0),
\qquad
M_0=u(z_0)=1,
\qquad
r_0=cM_0^{-1/N}.
\]
Given $z_j=(x_j,t_j)$ and $M_j=u(z_j)$, put
\begin{equation}\label{eq:doubling-cylinder}
\mathcal Q_j
:=
B_{4r_j}(x_j)
\times
\bigl(t_j-M_j^{2-p}(4r_j)^{sp},t_j\bigr],
\qquad
r_j:=cM_j^{-1/N}.
\end{equation}
If $u\le2M_j$ in $\mathcal Q_j$, the construction stops. Otherwise, by
continuity, we may choose $z_{j+1}\in\mathcal Q_j$ such that
$M_{j+1}:=u(z_{j+1})>2M_j$.

Let
\[
\alpha:=p-2+\frac{sp}{N}>0.
\]
As long as the construction continues, $M_j>2^j$, and hence
\begin{align*}
|x_j|
&\le
4c\sum_{i=0}^{j-1}M_i^{-1/N}
\le
4c\sum_{i=0}^\infty2^{-i/N},
\\
-t_j
&\le
4^{sp}c^{sp}
\sum_{i=0}^{j-1}M_i^{-\alpha}
\le
4^{sp}c^{sp}
\sum_{i=0}^\infty2^{-i\alpha}.
\end{align*}
We fix $c$ so small that these bounds, including one further cylinder as in
\eqref{eq:doubling-cylinder}, keep every $\mathcal Q_j$ inside
$B_{1/2}\times(-1/2,0]$. The iteration must stop after finitely many steps
because otherwise $M_j>2^j$ at points in a fixed compact subset of
$B_2\times(-1,1)$, contradicting the local boundedness of $u$.

Let $z_*=(x_*,t_*)$, $M=u(z_*)$, and $r=cM^{-1/N}$ be the quantities at the
stopping index. Thus $M\ge1$ and
\begin{equation}\label{eq:doubling-bound}
0\le u\le2M
\quad\text{in }
B_{4r}(x_*)
\times
\bigl(t_*-M^{2-p}(4r)^{sp},t_*\bigr].
\end{equation}
Define the intrinsic rescaling
\begin{equation}\label{eq:moment-rescaling}
v(x,t)
:=
\frac1M
u\bigl(x_*+rx,t_*+M^{2-p}r^{sp}t\bigr).
\end{equation}
The rescaled kernel
\[
K_*(x,y,t)
:=
r^N
K\bigl(x_*+rx,x_*+ry,t_*+M^{2-p}r^{sp}t\bigr)
\]
satisfies \eqref{eq:kernel} with the same constants. Consequently, $v$ is a
locally bounded weak solution in $B_4\times(-4^{sp},0]$, and
\begin{equation}\label{eq:rescaled-local-bounds}
0\le v\le2,
\qquad
v(0,0)=1.
\end{equation}

We compute
\begin{equation}\label{eq:rescaled-tail-identity}
\begin{split}
&\int_{-4^{sp}}^0
\int_{\R^n\setminus B_4}
\frac{|v(y,t)|^{p-1}}{|y|^N}\,dy\,dt
\\
&\qquad=
\frac1M
\int_{t_*-M^{2-p}(4r)^{sp}}^{t_*}
\int_{\R^n\setminus B_{4r}(x_*)}
\frac{|u(z,\tau)|^{p-1}}{|z-x_*|^N}\,dz\,d\tau.
\end{split}
\end{equation}
For $|z-x_*|>4r$, we claim that
\begin{equation}\label{eq:kernel-weight-comparison}
\frac1{M|z-x_*|^N}
\le
C_c\omega(z),
\end{equation}
where $C_c$ depends only on $c$ and the data. Indeed, if $|z|\le2$, then
\[
\frac1{M|z-x_*|^N}
\le
\frac1{M(4r)^N}
=
(4c)^{-N},
\]
whereas $\omega(z)\ge3^{-N}$. If $|z|>2$, then $|x_*|<1$ gives
$|z-x_*|\ge|z|/2$, and the claim follows from $M\ge1$ and
$|z|^{-N}\le C\omega(z)$. Since the time interval in
\eqref{eq:rescaled-tail-identity} lies in $[-1,0]$, we conclude that
\begin{equation}\label{eq:rescaled-tail-moment-bound}
\int_{-4^{sp}}^0
\int_{\R^n\setminus B_4}
\frac{|v(y,t)|^{p-1}}{|y|^N}\,dy\,dt
\le
C_c\mathcal M.
\end{equation}

Let $\delta_0$, $\rho_0$, and $c_0$ be the constants from Lemma
\ref{lem:continuity}. If $\mathcal M\le\delta_0/C_c$, then we get
that
\[
v\ge c_0
\quad\text{in }
B_{\rho_0}\times(-\rho_0^{sp},0].
\]
Returning to the original variables, we have
\begin{equation}\label{eq:physical-positive-cylinder}
u\ge c_0M
\end{equation}
in
\[
B_{\rho_0r}(x_*)
\times
\bigl(t_*-M^{2-p}(\rho_0r)^{sp},t_*\bigr].
\]
This cylinder is contained in $B_1\times[-1,0]$, where
$\omega\ge2^{-N}$. Therefore
\begin{align*}
\mathcal M
&\ge
2^{-N}c_0^{p-1}M^{p-1}
|B_{\rho_0r}|
M^{2-p}(\rho_0r)^{sp}
\\
&=
2^{-N}|B_1|c_0^{p-1}\rho_0^N
Mr^N
\\
&=
2^{-N}|B_1|c_0^{p-1}\rho_0^Nc^N
=:
\varepsilon_1>0.
\end{align*}
If instead $\mathcal M>\delta_0/C_c$, we already have a universal lower
bound. Thus \eqref{eq:absolute-moment} follows with
\[
\varepsilon_0
:=
\min\left\{\frac{\delta_0}{C_c},\varepsilon_1\right\}.
\]
\end{proof}

\begin{corollary}
\label{cor:positive-moment}
Let $u$ be a locally bounded weak solution in $B_2\times(-1,1)$ such that

\[
    u(0,0)=1,
    \qquad
    u\ge0
    \quad\text{a.e. in }B_4\times(-1,0].
\]

With $M_\pm$ as in \eqref{eq:weighted-moments}, we have
\begin{equation}\label{eq:moment-sum}
\int_{-1}^0\bigl(M_+(t)+M_-(t)\bigr)\,dt
\ge\varepsilon_0.
\end{equation}
Consequently, if
\begin{equation}\label{eq:small-neg-normalized}
\int_{-1}^0
\int_{\R^n\setminus B_4}
\frac{u_-(y,t)^{p-1}}{|y|^N}\,dy\,dt
\le\frac{\varepsilon_0}{2},
\end{equation}
then
\begin{equation}\label{eq:positive-moment-lower}
\int_{-1}^0M_+(t)\,dt
\ge\frac{\varepsilon_0}{2}.
\end{equation}
\end{corollary}

\begin{proof}
Since $u\ge0$ almost everywhere in $B_4\times(-1,0]$, the assumptions of
Proposition \ref{prop:moment} are satisfied. Moreover,
\[
|u|^{p-1}=u_+^{p-1}+u_-^{p-1}.
\]
Integrating this identity against $\omega$ over
$\R^n\times(-1,0)$ and applying Proposition \ref{prop:moment} gives
\eqref{eq:moment-sum}.

For the second assertion, the local nonnegativity implies that
$u_-=0$ almost everywhere in $B_4\times(-1,0]$. Since
\[
\omega(y)=(1+|y|)^{-N}\le |y|^{-N}
\quad\text{for }y\ne0,
\]
we obtain from \eqref{eq:small-neg-normalized} that
\begin{align*}
\int_{-1}^0M_-(t)\,dt
&=
\int_{-1}^0
\int_{\R^n\setminus B_4}
u_-(y,t)^{p-1}\omega(y)\,dy\,dt
\\
&\le
\int_{-1}^0
\int_{\R^n\setminus B_4}
\frac{u_-(y,t)^{p-1}}{|y|^N}\,dy\,dt
\le
\frac{\varepsilon_0}{2}.
\end{align*}
Subtracting this estimate from \eqref{eq:moment-sum} proves
\eqref{eq:positive-moment-lower}.
\end{proof}

\subsection{Persistence for an auxiliary Dirichlet problem}
We now show that we can evolve a cutoff for a short time with no
forcing term, while retaining a universal lower bound. Note that the 
operator is not assumed to be smooth. 
We fix a cutoff

\[
    0\le\eta\le1,
    \qquad
    \eta\in C_c^\infty(B_3),
    \qquad
    \eta=1\quad\text{on }B_2.
\]

For $\Theta>0$ and a symmetric measurable kernel $\widehat K(x,y,\theta)$
satisfying \eqref{eq:kernel} with $t$ replaced by $\theta$, we denote
the associated operator by $\mathcal L_{\widehat K,\theta}$ and let
$\Phi$ be the energy weak solution of
\begin{equation}\label{eq:Phi-problem}
\begin{cases}
\partial_\theta\Phi
+\mathcal{L}_{\widehat K,\theta}\Phi=0
&\text{in }B_3\times(0,\Theta),\\
\Phi=0
&\text{in }(\R^n\setminus B_3)\times(0,\Theta),\\
\Phi(\cdot,0)=\eta.
\end{cases}
\end{equation}

\begin{lemma}\label{lem:persistence}
For the solution $\Phi$ of \eqref{eq:Phi-problem}, there exist

\[
    \Theta_*>0,
    \qquad
    c_1\in(0,1),
\]

depending only on $n,p,s,\lambda,\Lambda$ such that
\begin{equation}\label{eq:Phi-bounds}
0\le\Phi\le1
\quad\text{a.e. in }\R^n\times(0,\Theta)
\end{equation}
and
\begin{equation}\label{eq:Phi-persistence}
\Phi\ge c_1
\quad\text{a.e. in }
B_1\times\bigl(0,\min\{\Theta,\Theta_*\}\bigr).
\end{equation}
\end{lemma}

\begin{proof}
Applying Lemma \ref{lem:weak-comparison} to $\Phi$ and the constant
solutions $0$ and $1$ gives \eqref{eq:Phi-bounds}.
We now prove \eqref{eq:Phi-persistence}.
We apply Lemma \ref{lem:APT-forward} with
\[
z_0=(0,0),
\qquad
\rho=1,
\qquad
M=2,
\qquad
\xi=\frac12,
\qquad
a=\frac12.
\]
The initial-positivity condition holds because
\[
\Phi(\cdot,0)=\eta=1=\xi M
\quad\text{in }B_2.
\]
Let $\nu_0$ be the constant supplied by Lemma
\ref{lem:APT-forward} for $a=1/2$, and set
\[
\vartheta_*
:=
\nu_0,
\qquad
\Theta_*:=2^{sp}\vartheta_*.
\]
For any
\[
0<\vartheta
<
\min\left\{\vartheta_*,\frac{\Theta}{2^{sp}}\right\},
\]
the relative measure of the sublevel set in the forward cylinder satisfies
\begin{align*}
\frac{
\bigl|\{\Phi<\xi M\}\cap Q^+_2((0,0);\vartheta)\bigr|
}{
\bigl|Q^+_2((0,0);\vartheta)\bigr|
}
&\le1
\\
&\le
\frac{\nu_0}{\vartheta}
=
\nu_0
\frac{(\xi M)^{2-p}}{\vartheta},
\end{align*}
because $\xi M=1$ and $\vartheta\le\vartheta_*=\nu_0$. Thus the measure
hypothesis of Lemma \ref{lem:APT-forward} is automatic, and hence
\[
\Phi\ge a\xi M=\frac12
\quad\text{a.e. in }
B_1\times(0,\vartheta2^{sp}].
\]
Letting $\vartheta$ increase to
$\min\{\vartheta_*,\Theta/2^{sp}\}$ proves
\eqref{eq:Phi-persistence} with $c_1=1/2$.
\end{proof}

\section{Construction of Barrier}

\subsection{A barrier from below}
We construct a barrier which starts from zero and whose positivity persists after the forcing term 
becomes zero.
For $f\in L^1(-1,0)$ with $f\ge0$, we extend $f$ by zero to $[0,2]$
and write
\begin{equation}\label{eq:b-def}
b(t):=\int_{-1}^t f(\sigma)\,d\sigma,
\qquad B:=\int_{-1}^0f(t)\,dt.
\end{equation}

\begin{proposition}\label{prop:bump}
There exists $B_0>0$, depending only on the data, such that for every
source $f$ as above with $0<B\le B_0$, there exists
\[
w\in C\bigl([-1,2];L^2(B_3)\bigr)
\cap L^p\bigl((-1,2);W^{s,p}_0(B_3)\bigr),
\qquad
w(x,t)=b(t)\Psi(x,t),
\]
such that

\[
    0\le\Psi\le1,
    \qquad
    \Psi(\cdot,t)=0
    \quad\text{on }\R^n\setminus B_3,
\]

and
\begin{equation}\label{eq:w-eq}
\partial_tw+\mathcal{L}_tw=f(t)\Psi
\end{equation}
weakly in $B_3\times(-1,2)$. Moreover,
\begin{equation}\label{eq:w-lower}
w\ge c_1B
\quad\text{a.e. in }B_1\times[0,2],
\end{equation}
where $c_1$ is the constant in Lemma \ref{lem:persistence}.
\end{proposition}

\begin{proof}
Let $\Theta_*$ be the constant in Lemma \ref{lem:persistence}, and choose
\[
B_0:=\min\left\{1,\left(\frac{\Theta_*}{6}\right)^{1/(p-2)}\right\}.
\]
The function $b$ is absolutely continuous and nondecreasing, with
$b(-1)=0$, $0\le b\le B$, and $b=B$ on $[0,2]$. Define
\begin{equation}\label{eq:bump-clock}
h(t):=\int_{-1}^t b(\sigma)^{p-2}\,d\sigma,
\qquad H:=h(2).
\end{equation}
Then
\begin{equation}\label{eq:bump-clock-bound}
0<H\le3B^{p-2}\le\frac{\Theta_*}{2}.
\end{equation}
Set
\[
t_*:=\sup\{t\in[-1,2]:b(t)=0\}.
\]
Since $b(0)=B>0$, we have $t_*<0$. Moreover, $h=0$ on
$[-1,t_*]$, whereas $h$ is strictly increasing on $(t_*,2]$, with
$h'(t)=b(t)^{p-2}>0$. We denote the inverse of $h$ in $(0,H)$ by $g$.
The maps $h:(t_*,2)\to(0,H)$ and
$g:(0,H)\to(t_*,2)$ are locally Lipschitz.

For $0<\theta<H$, define
\[
\widehat K(x,y,\theta):=K(x,y,g(\theta)),
\]
and extend it to $0<\theta<\Theta_*+1$ by setting
$\widehat K(x,y,\theta)=\lambda|x-y|^{-N}$ for $\theta\ge H$.
The resulting kernel is measurable, symmetric, and satisfies
\eqref{eq:kernel} with the same constants. 

Let $V:=W^{s,p}_0(B_3)$ and $p':=p/(p-1)$. Let $\Phi$ solve
\eqref{eq:Phi-problem} for $\widehat K$ on $(0,\Theta_*+1)$, with
initial value $\eta$. By \cite[Theorem~1.3 and
Proposition~2.23(ii)]{KaltenbachRuzicka}, applied to the Gelfand triple
$V\subset L^2(B_3)\subset V^*$, $\Phi$ exists and satisfies

\[
\Phi\in C\bigl([0,\Theta_*+1];L^2(B_3)\bigr)
\cap L^p\bigl((0,\Theta_*+1);V\bigr),
\qquad
\partial_\theta\Phi\in L^{p'}\bigl((0,\Theta_*+1);V^*\bigr).
\]
By Lemma \ref{lem:persistence}, $0\le\Phi\le1$ and
$\Phi\ge c_1$ almost everywhere in $B_1\times(0,\Theta_*)$.

Using the $L^2$-continuous representative of $\Phi$, set
\[
\Psi(\cdot,t):=\Phi(\cdot,h(t)),
\qquad w(\cdot,t):=b(t)\Psi(\cdot,t).
\]
Note that $\Psi=\eta$ and $w=0$ on $[-1,t_*]$. The bounds on
$\Psi$ and its zero exterior values follow from those of $\Phi$.
Also, $w\in C([-1,2];L^2(B_3))$. We compute
\begin{align*}
\int_{-1}^2\|w(\cdot,t)\|_V^p\,dt
&=\int_{t_*}^2 b(t)^p\|\Phi(\cdot,h(t))\|_V^p\,dt\\
&=\int_0^H b(g(\theta))^2\|\Phi(\cdot,\theta)\|_V^p\,d\theta\\
&\le B^2\int_0^H\|\Phi(\cdot,\theta)\|_V^p\,d\theta<\infty,
\end{align*}
to see that $w \in L^p\bigl((-1,2);V\bigr)$. 

In $B_3 \times (t_*,2)$ we have
\[
\partial_t\Psi+b(t)^{p-2}\mathcal L_t\Psi=0,
\]
and so
\[
\partial_tw
=f\Psi-b^{p-1}\mathcal L_t\Psi
=f\Psi-\mathcal L_tw.
\]
These identities are understood in the variational sense of
\eqref{eq:weak-form}. 

Note that $f\Psi\in L^1((-1,2);L^2(B_3))$ and
$\mathcal L_tw\in L^{p'}((-1,2);V^*)$ and so the right-hand side
$f\Psi-\mathcal L_tw$ belongs to $L^1((-1,2);V^*)$.
The continuity of $w$ in $L^2(B_3)$ and $w(\cdot,t_*)=0$ ensure
that extending the equation from $(t_*,2)$ to $t_*$ introduces no
boundary term in the weak formulation. Since $w=f=0$ almost
everywhere on $(-1,t_*)$, \eqref{eq:w-eq} follows throughout
$B_3\times(-1,2)$.

Finally, $h(0)>0$ and, by \eqref{eq:bump-clock-bound},
\[
0<h(0)\le h(t)\le H<\Theta_*
\qquad(0\le t\le2).
\]
For $t\in[0,2]$, we have $b(t)=B$ and
$h(t)=h(0)+B^{p-2}t$. Thus $h$ maps $[0,2]$ affinely onto
$[h(0),H]\subset(0,\Theta_*)$. Thus, the lower bound for
$\Phi$ gives
\[
w(x,t)=B\Phi(x,h(t))\ge c_1B
\quad\text{a.e. in }B_1\times[0,2],
\]
as required.
\end{proof}

\subsection{An equation for $u_+$}
Let $u$ be a locally bounded local weak solution of
\eqref{eq:main-equation} in $B_4\times(-1,2)$.
Suppose that, for some $\mathcal R\ge4$,
\begin{equation}\label{eq:local-nonnegative-normalized}
u\ge0
\quad\text{in }B_{\mathcal R}\times(-1,2).
\end{equation}
We define

\[
    U:=u_+.
\]
The negative tail and the remainder associated with truncation
are denoted by
\begin{equation}\label{eq:Tminus}
T_-(t)
:=
\int_{\R^n\setminus B_{\mathcal R}}
\frac{u_-(y,t)^{p-1}}{|y|^N}\,dy
\end{equation}
and, for $x\in B_3$, by
\begin{equation}\label{eq:R-def}
\begin{split}
\mathcal R_u(x,t)
:=
\int_{\{y:\,u(y,t)<0\}}
&\left[
\bigl(U(x,t)+u_-(y,t)\bigr)^{p-1}
-
U(x,t)^{p-1}
\right]
K(x,y,t)\,dy.
\end{split}
\end{equation}

\begin{lemma}\label{lem:remainder}
Under the preceding assumptions, $U$ satisfies
\begin{equation}\label{eq:U-equation}
\partial_tU+\mathcal{L}_tU=-\mathcal R_u
\end{equation}
weakly in $B_3\times(-1,2)$, with $\mathcal R_u\ge0$.
Moreover, for every $B>0$, whenever $x\in B_3$ and
$0\le U(x,t)\le B$,
\begin{equation}\label{eq:R-est}
\mathcal R_u(x,t)
\le
C_0\bigl(B^{p-1}+T_-(t)\bigr),
\end{equation}
where $C_0$ depends only on the data.
\end{lemma}

\begin{proof}
Note that since $U=u$ in $B_{\mathcal R}$ and $|U|\le|u|$ globally, $U$ has
the local regularity and tail integrability required in
Definition \ref{def:weak} on $B_3\times(-1,2)$.
The set $\{y:u(y,t)<0\}$ is contained in
$\R^n\setminus B_{\mathcal R}$. For $x\in B_3$ and $|y|\ge\mathcal R$,
\[
|x-y|\ge |y|-3\ge\frac{|y|}{4}.
\]
Moreover, for $a,z\ge0$,
\[
0\le(a+z)^{p-1}-a^{p-1}
\le C\bigl(a^{p-1}+z^{p-1}\bigr).
\]
The upper kernel bound therefore gives
\begin{align*}
0\le\mathcal R_u(x,t)
&\le C\int_{\R^n\setminus B_{\mathcal R}}
\frac{U(x,t)^{p-1}+u_-(y,t)^{p-1}}{|y|^N}\,dy\\
&\le C\bigl(U(x,t)^{p-1}+T_-(t)\bigr),
\end{align*}
where we used $\mathcal R\ge4$ and
$\int_{|y|\ge4}|y|^{-n-sp}\,dy<\infty$.
This proves both $\mathcal R_u\ge0$ and \eqref{eq:R-est}.
It also shows that, for every compact interval $J\Subset(-1,2)$,
\[
\mathcal R_u\in L^1\bigl(J;L^2(B_3)\bigr),
\]
by the local boundedness of $u$ and \eqref{eq:weak-tail-condition}.

To prove \eqref{eq:U-equation}, fix $[t_1,t_2]\Subset(-1,2)$ and
an admissible test function $\varphi$ supported
in $B_3$.
Since $U=u$ in $B_3$, only interactions with the set
$\{y:u(y,t)<0\}$ contribute to the difference of the energy forms.
By symmetry of $K$ and the definition of $\mathcal R_u$ we get,
\begin{align*}
\mathcal E_t(u,\varphi)-\mathcal E_t(U,\varphi)
&=\int_{B_3}\int_{\{y:u(y,t)<0\}}
\Bigl[\bigl(U(x,t)+u_-(y,t)\bigr)^{p-1}
-U(x,t)^{p-1}\Bigr]\\
&\hspace{3em}\times K(x,y,t)\varphi(x,t)\,dy\,dx\\
&=\int_{B_3}\mathcal R_u(x,t)\varphi(x,t)\,dx.
\end{align*}
We now integrate over $(t_1,t_2)$ and use the weak
formulation for $u$, together with $U=u$ in $B_3$, to obtain
\[
\left.\int_{B_3}U\varphi\,dx\right|_{t_1}^{t_2}
-\int_{t_1}^{t_2}\int_{B_3}U\,\partial_t\varphi\,dx\,dt
+\int_{t_1}^{t_2}\mathcal E_t(U,\varphi)\,dt
=-\int_{t_1}^{t_2}\int_{B_3}\mathcal R_u\varphi\,dx\,dt,
\]
which is \eqref{eq:U-equation}. 
\end{proof}

\subsection{Comparison against the barrier}
We retain the assumptions and notation of the preceding subsection,
and use $M_+$ from \eqref{eq:weighted-moments}; thus
$M_+(t)=\int_{\R^n}U(y,t)^{p-1}\omega(y)\,dy$.
We also set
\[
    D
    :=
    2^{p-1}\int_{\R^n}\omega(y)\,dy.
\]
For a barrier height $B>0$, we define the correction
\begin{equation}\label{eq:ell-def}
\ell(t)
:=
C_0
\int_{-1}^t
\bigl(B^{p-1}+T_-(\sigma)\bigr)\,d\sigma,
\end{equation}
where $C_0$ is the constant from Lemma \ref{lem:remainder}.

\begin{proposition}\label{prop:forced-comparison}
Under the preceding assumptions on $u$, let $0<B<1$ with $B\le B_0$,
$w=b\Psi$, and $f$ be as in Proposition \ref{prop:bump}.
There exists $\kappa>0$, depending only on the data, such that if
\begin{equation}\label{eq:f-condition}
0\le f(t)\le\kappa\bigl[M_+(t)-DB^{p-1}\bigr]_+,
\end{equation}
then
\begin{equation}\label{eq:comparison-final}
U(x,t)\ge w(x,t)-\ell(t)
\end{equation}
for almost every $(x,t)\in B_3\times(-1,2)$.
\end{proposition}

\begin{proof}
It suffices to consider the case in which $\ell$ is finite.
Then $\ell$ is nonnegative and locally absolutely continuous, with
$\ell'=C_0(B^{p-1}+T_-)$. Set
\[
d:=w-U-\ell,\qquad z:=d_+.
\]
Since $U\ge0$ globally and $w=0$ outside $B_3$, we have
\begin{equation}\label{eq:comparison-z-bounds}
z=0\quad\text{outside }B_3,
\qquad 0\le z\le w\le b(t)\le B.
\end{equation}
Moreover, on $\{z>0\}$ we have $0\le U<w-\ell\le B$, and hence
Lemma \ref{lem:remainder} gives
\begin{equation}\label{eq:comparison-remainder-cancel}
\bigl(\mathcal R_u-\ell'\bigr)z\le0.
\end{equation}

For almost every fixed $t$, we suppress the time variable and set
\[
A:=\{x:z(x)>0\},\qquad G:=\{y:U(y)>2B\}.
\]
We seek a lower bound for
$\mathcal E_t(w,z)-\mathcal E_t(U,z)$. Since
\[
\bigl(w(x)-w(y)\bigr)-\bigl(U(x)-U(y)\bigr)=d(x)-d(y),
\]
the monotonicity of $F$ and of the map $z \mapsto z_+$ implies
\[
\Bigl[F\bigl(w(x)-w(y)\bigr)-F\bigl(U(x)-U(y)\bigr)\Bigr]
\bigl(z(x)-z(y)\bigr)\ge0.
\]
So the integrand defining
$\mathcal E_t(w,z)-\mathcal E_t(U,z)$ is nonnegative and
we may therefore obtain a lower bound by restricting the double
integral to $(A\times G)\cup(G\times A)$.
Now, $A\subset B_3$, $A\cap G=\varnothing$, and
for $x\in A$ and $y\in G$ we have
\[
d(x)-d(y)
=w(x)-U(x)+U(y)-w(y)
\ge U(y)-B\ge\frac{U(y)}2.
\]
Lemma \ref{lem:scalar} therefore yields
\[
F\bigl(w(x)-w(y)\bigr)-F\bigl(U(x)-U(y)\bigr)
\ge 2^{3-2p}U(y)^{p-1}.
\]
Also, for $x\in B_3$,
\[
K(x,y,t)\ge\lambda(3+|y|)^{-N}
\ge\lambda3^{-N}\omega(y).
\]
Keeping only the interactions in $A\times G$ and $G\times A$
and using symmetry, we obtain
\[
\mathcal E_t(w,z)-\mathcal E_t(U,z)
\ge\kappa
\left(\int_G U(y)^{p-1}\omega(y)\,dy\right)
\int_{B_3}z(x)\,dx,
\]
where $\kappa:=2^{3-2p}\lambda3^{-N}$. Since $U\le2B$ on
$\R^n\setminus G$, the definition of $D$ gives
\[
\int_G U(y)^{p-1}\omega(y)\,dy
\ge\bigl[M_+(t)-DB^{p-1}\bigr]_+.
\]
Restoring the time variable and applying \eqref{eq:f-condition},
we conclude that
\begin{equation}\label{eq:comparison-energy-gain}
\mathcal E_t(w,z)-\mathcal E_t(U,z)
\ge f(t)\int_{B_3}z(x,t)\,dx.
\end{equation}

We now show that $z=(w-U-\ell)_+$ vanishes. Subtracting
\eqref{eq:U-equation} from \eqref{eq:w-eq} and using
$d=w-U-\ell$, we obtain
\[
\partial_t d+\mathcal L_t w-\mathcal L_t U
=f\Psi+\mathcal R_u-\ell'
\quad\text{in }B_3\times(-1,2)
\]
in the weak sense.

\iffalse 
We justify testing with $z=d_+$ on a compact time interval
$[a,t]\Subset(-1,2)$. Put $V=W^{s,p}_0(B_3)$ and
$p'=p/(p-1)$. The positive-part truncation controls the fractional
seminorm of $z$ on $B_3\times B_3$, while $0\le z\le w$ controls
its interactions with $\R^n\setminus B_3$. Hence
$z\in L^p((a,t);V)\cap C([a,t];L^2(B_3))$.
Choose $0\le\chi\le1$ in $C_c^\infty(B_4)$, equal to $1$ on a
neighbourhood of $\overline{B_3}$. The function
$\widetilde d=w-\chi(U+\ell)$ belongs locally in time to
$L^p(W^{s,p}(\R^n))\cap C(L^2(\R^n))$, agrees with $d$ on $B_3$,
and is nonpositive outside $B_3$. Its Steklov averages therefore
have positive parts $z_h$ that are admissible tests, vanish outside
$B_3$, and converge to $z$ in $L^p((a,t);V)$ and
$C([a,t];L^2(B_3))$.

For completeness, $\mathcal L_\sigma(\chi U)$ belongs to
$L^{p'}((a,t);V^*)$. On $B_3$, the difference
$\mathcal L_\sigma U-\mathcal L_\sigma(\chi U)$ is an exterior
interaction separated from the diagonal, and belongs to
$L^1((a,t);L^2(B_3))$ by local boundedness and the tail condition.
The terms $f\Psi$, $\mathcal R_u$, and $\ell'$ have the latter
integrability as well. Thus we may average the equation in time,
test with $z_h$, and pass to the limit: the energy convergence
handles the $L^{p'}(V^*)$ terms and the uniform $L^2$ convergence
handles the $L^1(L^2)$ terms. Here the actual time-dependent flux
is averaged; it is not replaced by the flux of the averaged
solution. The time term gives the difference of the squared
$L^2$ norms. 
\fi 
Testing with $z=d_+$, we obtain
\begin{align*}
\frac12\|z(t)\|_{L^2(B_3)}^2
-\frac12\|z(a)\|_{L^2(B_3)}^2
&+\int_a^t\bigl[\mathcal E_\sigma(w,z)
-\mathcal E_\sigma(U,z)\bigr]\,d\sigma\\
&=\int_a^t\int_{B_3}
\bigl(f\Psi+\mathcal R_u-\ell'\bigr)z\,dx\,d\sigma\\
&\le\int_a^t f(\sigma)\int_{B_3}z(x,\sigma)\,dx\,d\sigma,
\end{align*}
where we used \eqref{eq:comparison-remainder-cancel} and $0\le\Psi\le1$.
From \eqref{eq:comparison-energy-gain} and 
\eqref{eq:comparison-z-bounds}, we therefore obtain
\[
\|z(t)\|_{L^2(B_3)}^2
\le\|z(a)\|_{L^2(B_3)}^2
\le |B_3|\,b(a)^2.
\]
Letting $a\downarrow-1$ and using $b(-1)=0$, we obtain $z(t)=0$.
Therefore $U\ge w-\ell$, as claimed.
\end{proof}

\section{Proof of the Main Theorem}

We first prove a normalised version of the main theorem. 
\begin{proposition}\label{prop:normalized}
Suppose $u$ is a locally bounded weak solution in

\[
    B_4\times(-2,3),
\]

with

\[
    u(0,0)=1.
\]

Suppose moreover that, for some $\mathcal R\ge4$,

\[
    u\ge0
    \quad\text{in }B_{\mathcal R}\times(-2,3).
\]

There exist universal constants $c_*>0$ and $\delta_*>0$ such that either
\begin{equation}\label{eq:tail-alternative}
\int_{-1}^2
\int_{\R^n\setminus B_{\mathcal R}}
\frac{u_-(y,t)^{p-1}}{|y|^N}\,dy\,dt
\ge\delta_*,
\end{equation}
or
\begin{equation}\label{eq:positive-alternative}
\operatorname*{ess\,inf}_{B_1\times(0,2)}u
\ge c_*.
\end{equation}
\end{proposition}

\begin{proof}
Let $\varepsilon_0$ be the constant in Proposition \ref{prop:moment} and Corollary
\ref{cor:positive-moment}, and let $B_0,c_1,\kappa,D,C_0$ be as in
Propositions \ref{prop:bump} and \ref{prop:forced-comparison}.
Since $p>2$, we may fix $0<B<1$, depending only on the data, such that
\[
B\le B_0,\qquad
DB^{p-1}\le\frac{\varepsilon_0}{4},\qquad
B\le\frac{\kappa\varepsilon_0}{4},\qquad
3C_0B^{p-2}\le\frac{c_1}{4}.
\]
Set
\[
\delta_*:=\min\left\{
\frac{\varepsilon_0}{2},\frac{c_1B}{4C_0}
\right\},
\qquad c_*:=\frac{c_1B}{2}.
\]
Suppose that \eqref{eq:tail-alternative} fails. With $T_-$ defined
by \eqref{eq:Tminus}, this means that
\[
\int_{-1}^2T_-(t)\,dt<\delta_*.
\]

We use $U=u_+$ and the moment $M_+$ from \eqref{eq:weighted-moments}.
Since $u_-=0$ in $B_{\mathcal R}$ and $\mathcal R\ge4$, we have
\[
\int_{-1}^0\int_{\R^n\setminus B_4}
\frac{u_-(y,t)^{p-1}}{|y|^N}\,dy\,dt
=\int_{-1}^0T_-(t)\,dt
<\delta_*\le\frac{\varepsilon_0}{2}.
\]
Corollary \ref{cor:positive-moment} therefore yields
\[
\int_{-1}^0M_+(t)\,dt\ge\frac{\varepsilon_0}{2}.
\]
We use this estimate to choose $f$ satisfying both
\eqref{eq:f-condition} and $\int_{-1}^0f=B$. Define
\[
q(t):=\kappa\bigl[M_+(t)-DB^{p-1}\bigr]_+,
\qquad Q:=\int_{-1}^0q(t)\,dt.
\]
We have, 
\[
Q\ge\kappa\left(\int_{-1}^0M_+(t)\,dt-DB^{p-1}\right)
\ge\frac{\kappa\varepsilon_0}{4}\ge B.
\]
Thus
\[
f(t):=\frac{B}{Q}q(t)\quad(-1<t<0),
\qquad f(t):=0\quad(0\le t\le2)
\]
is nonnegative, belongs to $L^1(-1,0)$, and satisfies
\[
\int_{-1}^0f(t)\,dt=B,
\qquad
f(t)\le\kappa\bigl[M_+(t)-DB^{p-1}\bigr]_+
\quad\text{for a.e. }t\in(-1,2).
\]

We apply Proposition \ref{prop:bump} with this $f$ to obtain $w$, and
then Proposition \ref{prop:forced-comparison} to obtain $U\ge w-\ell$.
For $-1<t<2$, the choices of $B$ and $\delta_*$ give
\[
\ell(t)
\le C_0\left(3B^{p-1}+\int_{-1}^2T_-(\sigma)\,d\sigma\right)
\le\frac{c_1B}{4}+C_0\delta_*
\le\frac{c_1B}{2}.
\]
Since $u=U$ in $B_1\times(-1,2)$ and
$w\ge c_1B$ almost everywhere in $B_1\times(0,2)$, it follows that
\[
u\ge w-\ell\ge\frac{c_1B}{2}=c_*
\quad\text{a.e. in }B_1\times(0,2).
\]
In particular, \eqref{eq:positive-alternative} holds.
\end{proof}

\begin{proof}[Proof of Theorem \ref{thm:main}]
Set $k=u(x_0,t_0)>0$ and $\tau=k^{2-p}r^{sp}$, and define
\[
v(x,t):=\frac{u(x_0+rx,t_0+\tau t)}{k},
\qquad \mathcal R:=\frac Rr\ge4.
\]
The rescaled kernel
\[
\widetilde K(x,y,t)
:=r^{n+sp}K(x_0+rx,x_0+ry,t_0+\tau t)
\]
is symmetric and satisfies \eqref{eq:kernel} with the same constants
$\lambda$ and $\Lambda$. Since $\tau k^{p-2}=r^{sp}$, a change of
variables in the weak formulation shows that $v$ is a locally bounded
local weak solution of
\[
\partial_t v+\widetilde{\mathcal L}_t v=0
\quad\text{in }B_4\times(-2,3),
\]
where $\widetilde{\mathcal L}_t$ is the operator associated with
$\widetilde K$. The local energy regularity and the tail integrability
in Definition \ref{def:weak} are preserved by this change of variables.
Moreover, $v(0,0)=1$ and $v\ge0$ in
$B_{\mathcal R}\times(-2,3)$. Thus Proposition \ref{prop:normalized}
applies to $v$.

If the positive alternative holds, then
\[
\operatorname*{ess\,inf}_{B_r(x_0)\times(t_0,t_0+2\tau)}u
=k\operatorname*{ess\,inf}_{B_1\times(0,2)}v
\ge kc_*.
\]
Consequently,
\[
k\le c_*^{-1}
\operatorname*{ess\,inf}_{B_r(x_0)\times(t_0,t_0+2\tau)}u.
\]

Otherwise, the tail alternative holds. With
$J=(t_0-\tau,t_0+2\tau)$, changing variables
$Y=x_0+ry$ and $\sigma=t_0+\tau t$ gives
\begin{align*}
\delta_*
&\le\int_{-1}^2\int_{\R^n\setminus B_{\mathcal R}}
\frac{v_-(y,t)^{p-1}}{|y|^{n+sp}}\,dy\,dt\\
&=\frac{r^{sp}}{\tau k^{p-1}}
\int_J\int_{\R^n\setminus B_R(x_0)}
\frac{u_-(Y,\sigma)^{p-1}}{|Y-x_0|^{n+sp}}\,dY\,d\sigma\\
&=\frac{3}{k^{p-1}}\left(\frac rR\right)^{sp}
\operatorname{Tail}_{p-1}(u_-;x_0,R,J)^{p-1},
\end{align*}
where the last equality follows from \eqref{eq:tail-def} and
$|J|=3\tau$. Hence
\[
k\le\left(\frac{3}{\delta_*}\right)^{1/(p-1)}
\left(\frac rR\right)^{sp/(p-1)}
\operatorname{Tail}_{p-1}(u_-;x_0,R,J).
\]
Both terms on the right-hand side of \eqref{eq:main-harnack} are
nonnegative. Therefore either alternative implies
\eqref{eq:main-harnack} with
\[
C:=\max\left\{c_*^{-1},
\left(\frac{3}{\delta_*}\right)^{1/(p-1)}\right\},
\]
which depends only on the data.
\end{proof}

% =============================================================================
\appendix
\section{The Linear Case}\label{app:linear}

We now explain how the comparison argument can be adapted to $p=2$.
Throughout this appendix, $N=n+2s$, $F(a)=a$, and the kernel satisfies
\eqref{eq:kernel} with $p=2$. Weak solutions and tails are understood
according to Definition \ref{def:weak} and \eqref{eq:tail-def}, with
the same substitution. All constants depend only on
$n,s,\lambda,\Lambda$, unless stated otherwise.

There are two changes to the proof. First, the time reparametrisation defined in
\eqref{eq:bump-clock} no longer slows down when the barrier height
decreases: for $p=2$, it is simply elapsed time. We therefore construct
the barrier on a short, fixed time interval and subsequently iterate
the resulting estimate in time. Second, linearity gives the sharper
truncation remainder bound $\mathcal R_u\le C T_-$, without the term involving the
barrier height in \eqref{eq:R-est}. These changes replace the two uses of small powers
of the barrier height in the degenerate argument.

\begin{theorem}[Linear Harnack inequality]\label{thm:linear-harnack}
Let $u$ be a locally bounded local weak solution of
$\partial_tu+\mathcal L_tu=0$ with $p=2$. Given $r>0$ and $(x_0,t_0)$,
set
\[
\tau=r^{2s},\qquad J=(t_0-\tau,t_0+2\tau).
\]
Assume that the solution domain contains
\[
B_{4r}(x_0)\times(t_0-2\tau,t_0+3\tau),
\]
and that, for some $R\ge4r$,
\[
u\ge0\quad\text{a.e. in }
B_R(x_0)\times(t_0-2\tau,t_0+3\tau).
\]
Then,
\begin{equation}\label{eq:linear-harnack}
u(x_0,t_0)
\le C\operatorname*{ess\,inf}_{B_r(x_0)\times(t_0,t_0+2\tau)}u
+C\left(\frac rR\right)^{2s}
\operatorname{Tail}_1(u_-;x_0,R,J).
\end{equation}
\end{theorem}

\begin{remark}[Immediate forward comparison in the linear case]
\label{rem:linear-no-gap}
We get the forward
comparison without an additional gap while retaining the backward
history and tail interval in the theorem in the linear case too.
The same-time consequence is stated in
Corollary \ref{cor:linear-elliptic-type} below.
Restricting the comparison interval to $(t_0+\tau,t_0+2\tau)$
recovers the separated-time form of the linear Harnack estimate
with a time-integrated tail; see \cite{KassmannWeidner}.
The proof treats $p=2$ directly and does not assert uniformity of
the nonlinear constants as $p\downarrow2$.
\end{remark}

\begin{corollary}[Linear elliptic-type Harnack inequality]
\label{cor:linear-elliptic-type}
Under the assumptions of Theorem \ref{thm:linear-harnack}, with
$\tau=r^{2s}$ and $J=(t_0-\tau,t_0+2\tau)$, we have
\begin{equation}\label{eq:linear-same-time-harnack}
u(x_0,t_0)
\le C\inf_{B_r(x_0)}u(\cdot,t_0)
+C\left(\frac rR\right)^{2s}
\operatorname{Tail}_1(u_-;x_0,R,J),
\end{equation}
where $C$ depends only on $n,s,\lambda,\Lambda$. In particular, if
$u\ge0$ on $\R^n\times(t_0-2\tau,t_0+3\tau)$, then
\[
u(x_0,t_0)\le C\inf_{B_r(x_0)}u(\cdot,t_0).
\]

\end{corollary}

\begin{proof}
Apply \eqref{eq:linear-harnack} and pass to the reference time by
continuity, as in the proof of Corollary \ref{cor:elliptic-type}.
\end{proof}

\begin{example}[The local heat equation]\label{ex:local-heat}
For $a>0$, the function
\[
v_a(x,t)=\exp(a x_1+a^2t)
\]
is a positive solution of $(v_a)_t-\Delta v_a=0$ on
$\R^n\times\R$, with $v_a(0,0)=1$. For any fixed $r>0$,
\[
\operatorname*{ess\,inf}_{B_r\times(0,2r^2)}v_a=e^{-ar}.
\]
Letting $a\to\infty$ rules out a universal forward Harnack estimate
without a time gap, even with arbitrarily long backward history.
In contrast, at the later time $r^2$,
\[
\inf_{B_r}v_a(\cdot,r^2)
=e^{a^2r^2-ar}
=e^{(ar-1/2)^2-1/4}\ge e^{-1/4}.
\]
Thus the obstruction in the local theory is not confined to the
finite propagation exhibited in Example \ref{ex:local-p-wave}.
\end{example}

\begin{proof}[Proof of Theorem \ref{thm:linear-harnack}]
We first work in $B_4\times(-2,3)$ and suppose that $u\ge0$ in
$B_{\mathcal R}\times(-2,3)$ for some $\mathcal R\ge4$.
We set
\[
\omega(y)=(1+|y|)^{-N},\qquad
M_\pm(t)=\int_{\R^n}u_\pm(y,t)\omega(y)\,dy,
\qquad
T_-(t)=\int_{\R^n\setminus B_{\mathcal R}}
\frac{u_-(y,t)}{|y|^N}\,dy.
\]
These quantities are integrable on compact time intervals, and
$M_-(t)\le T_-(t)$ by local nonnegativity.

\smallskip
\noindent\emph{Step 1: a point-to-moment estimate on a short interval.}
For every fixed $a\in(0,1)$, the proof of Proposition
\ref{prop:moment} gives
\begin{equation}\label{eq:linear-point-moment}
u(0,0)\le C_a\int_{-a}^0\bigl(M_+(t)+M_-(t)\bigr)\,dt,
\end{equation}
where $C_a$ depends additionally on $a$. We give the changes needed
to obtain this estimate below. 

If $u(0,0)=0$, the estimate is immediate. Otherwise, linearity
allows us to divide by $u(0,0)$ and assume that $u(0,0)=1$.
In the doubling construction, use
\[
r_j=c_aM_j^{-1/N},\qquad
\mathcal Q_j=B_{4r_j}(x_j)\times
\bigl(t_j-(4r_j)^{2s},t_j\bigr].
\]
The two series are bounded by
\[
4c_a\sum_{j=0}^\infty2^{-j/N}
\quad\text{and}\quad
4^{2s}c_a^{2s}\sum_{j=0}^\infty2^{-2sj/N},
\]
respectively. We choose $c_a>0$ small enough that all
the cylinders stay in $B_{1/2}\times(-a/2,0]$. As before,
local boundedness forces the construction to stop. At the stopping
point $(x_*,t_*)$, with height $M\ge1$ and $r=c_aM^{-1/N}$, we set
\[
v(x,t)=M^{-1}u(x_*+rx,t_*+r^{2s}t).
\]
Then $v(0,0)=1$ and $0\le v\le2$ in
$B_4\times(-4^{2s},0]$. Writing
\[
\mathcal M_a=\int_{-a}^0\bigl(M_+(t)+M_-(t)\bigr)\,dt,
\]
the tail computation in \eqref{eq:rescaled-tail-identity}--
\eqref{eq:rescaled-tail-moment-bound} becomes
\[
\int_{-4^{2s}}^0\int_{\R^n\setminus B_4}
\frac{|v(y,t)|}{|y|^N}\,dy\,dt
\le C_{c_a}\mathcal M_a.
\]

The proof of Lemma \ref{lem:continuity} applies at $p=2$
Hence, if $\mathcal M_a$ is
sufficiently small, $v\ge c_0$ in a fixed backward cylinder.
Returning to the original variables gives
\[
\mathcal M_a\ge cMr^{n+2s}=cc_a^N>0.
\]
Together with the alternative that $\mathcal M_a$ is not small,
this proves a positive lower bound depending only on $a$ and the
data. 
\smallskip

\noindent\emph{Step 2: a short-time barrier.}
Let $\eta$ be the cutoff used in the auxiliary Dirichlet problem.
The proof of Lemma \ref{lem:persistence} also applies at $p=2$.
Indeed, the energy and forward De Giorgi argument in
\cite[Lemma~3.11]{APT} remains valid at this exponent; it is used
here only for a globally bounded, globally nonnegative auxiliary
solution with initial value $1$ on $B_2$. Thus there are
$\Theta_*>0$ and $c_1\in(0,1)$, depending only on the data, for
which the same conclusion holds.

Choose $a>0$ with $3a<\Theta_*$, and let $\Psi$ solve
\[
\begin{cases}
\partial_t\Psi+\mathcal L_t\Psi=0
&\text{in }B_3\times(-a,2a),\\
\Psi=0&\text{in }(\R^n\setminus B_3)\times(-a,2a),\\
\Psi(\cdot,-a)=\eta.
\end{cases}
\]
Then we have
\[
0\le\Psi\le1,\qquad
\Psi\ge c_1\quad\text{a.e. in }B_1\times(-a,2a).
\]
For $f\in L^1(-a,0)$, $f\ge0$, extended by zero to $(0,2a)$,
define
\[
b(t)=\int_{-a}^t f(\sigma)\,d\sigma,\qquad B=b(0),\qquad w=b\Psi.
\]
The product rule and linearity yield
\begin{equation}\label{eq:linear-barrier}
\partial_tw+\mathcal L_tw=f\Psi,\qquad
0\le w\le B,\qquad
w\ge c_1B\quad\text{in }B_1\times(0,2a).
\end{equation}
These are variational identities, with
\[
w\in C([-a,2a];L^2(B_3))
\cap L^2((-a,2a);W^{s,2}_0(B_3)).
\]
The source term satisfies
\[
f\Psi\in L^1((-a,2a);L^2(B_3)),
\]
so no stronger time integrability of $f$ is required.
Unlike Proposition \ref{prop:bump}, this construction imposes no
smallness condition on $B$.

\smallskip
\noindent\emph{Step 3: the remainder and the choice of forcing.}
For $U=u_+$, the identity in Lemma \ref{lem:remainder} now reads
\[
\partial_tU+\mathcal L_tU=-\mathcal R_u,\qquad
\mathcal R_u(x,t)=\int_{\{u(y,t)<0\}}u_-(y,t)K(x,y,t)\,dy.
\]
For $x\in B_3$ and $|y|\ge\mathcal R\ge4$, we have
$|x-y|\ge|y|/4$, and consequently
\[
0\le\mathcal R_u(x,t)\le C_0T_-(t).
\]
We therefore replace \eqref{eq:ell-def} by
\begin{equation}\label{eq:linear-correction}
\ell(t)=C_0\int_{-a}^tT_-(\sigma)\,d\sigma.
\end{equation}
This is the second modification: there is no term proportional to
$B$ in the correction.

Set
\[
D=2\int_{\R^n}\omega(y)\,dy,\qquad
\kappa=\frac{\lambda}{2\,3^N}.
\]
The proof of Proposition \ref{prop:forced-comparison} gives
$U\ge w-\ell$ whenever
\begin{equation}\label{eq:linear-forcing-condition}
0\le f(t)\le\kappa[M_+(t)-DB]_+.
\end{equation}

We impose $\kappa Da\le1$, and write
\[
I=\int_{-a}^0M_+(t)\,dt.
\]
If $I>0$, choose
\[
B=\frac{\kappa I}{2},\qquad
q(t)=\kappa[M_+(t)-DB]_+,\qquad
Q=\int_{-a}^0q(t)\,dt.
\]
Then
\[
Q\ge\kappa(I-aDB)
=\kappa I\left(1-\frac{\kappa Da}{2}\right)
\ge\frac{\kappa I}{2}=B.
\]
Thus $f=(B/Q)q$ on $(-a,0)$, extended by zero afterwards,
satisfies \eqref{eq:linear-forcing-condition} and has integral $B$.
By \eqref{eq:linear-barrier} and \eqref{eq:linear-correction},
\begin{equation}\label{eq:linear-moment-lower}
\operatorname*{ess\,inf}_{B_1\times(0,2a)}u
\ge\frac{c_1\kappa}{2}I-C_0\int_{-a}^{2a}T_-(t)\,dt.
\end{equation}
For $I=0$ this inequality follows from $u\ge0$ locally.
Combining it with \eqref{eq:linear-point-moment} and $M_-\le T_-$
gives
\begin{equation}\label{eq:linear-short-harnack}
u(0,0)\le
A\operatorname*{ess\,inf}_{B_1\times(0,2a)}u
+C\int_{-a}^{2a}T_-(t)\,dt,
\end{equation}
where $A\ge1$ and $C$ depend only on $a$ and the data.
The argument uses the equation in $B_4$ and nonnegativity in
$B_{\mathcal R}$ only on a neighbourhood of $[-a,2a]$.
In particular, the same estimate applies after translating this interval.

\smallskip
\noindent\emph{Step 4: iteration in time and scaling.}
Fix an integer $m\ge2$, depending only on the data, so large that
$a=1/m$ satisfies $3a<\Theta_*$ and $\kappa Da\le1$.
We set
\[
t_j=ja,\qquad k_j=u(0,t_j),\qquad
E=\int_{-1}^2T_-(t)\,dt,
\]
and apply \eqref{eq:linear-short-harnack} at $(0,t_j)$ for
$j=0,\ldots,2m-2$. Every tail interval
$((j-1)a,(j+2)a)$ lies in $(-1,2)$, and the translated local
cylinders lie in the assumed solution domain. Hence
\[
k_j\le A\operatorname*{ess\,inf}_{B_1\times(t_j,t_{j+2})}u+CE.
\]
By continuity at $(0,t_{j+1})$, this also gives
$k_j\le Ak_{j+1}+CE$. Iterating, and then using the estimate on
the $j$th cylinder, yields
\[
k_0\le A^{j+1}u(x,t)+CE\sum_{i=0}^jA^i
\quad\text{for a.e. }(x,t)\in
B_1\times(t_j,t_{j+2}).
\]
The intervals $(t_j,t_{j+2})$ with $j=0,\ldots,2m-2$ cover $(0,2)$.
Since $u\ge0$ there, taking a common constant
and then the essential infimum proves
\[
u(0,0)\le C\operatorname*{ess\,inf}_{B_1\times(0,2)}u+CE.
\]
Finally, we apply the normalised estimate to
$v(x,t)=u(x_0+rx,t_0+r^{2s}t)$ and set $\mathcal R=R/r$.
The rescaled kernel has the same ellipticity constants. A change
of variables gives
\begin{align*}
\int_{-1}^2\int_{\R^n\setminus B_{\mathcal R}}
\frac{v_-(y,t)}{|y|^{n+2s}}\,dy\,dt
&=\int_J\int_{\R^n\setminus B_R(x_0)}
\frac{u_-(Y,\sigma)}{|Y-x_0|^{n+2s}}\,dY\,d\sigma\\
&=3\left(\frac rR\right)^{2s}
\operatorname{Tail}_1(u_-;x_0,R,J),
\end{align*}
which proves \eqref{eq:linear-harnack}.
\end{proof}

% =============================================================================

\end{document}